\documentclass[12pt]{article}
\usepackage{graphicx}
\usepackage{amsmath,amsthm,amssymb,enumerate}
\usepackage{euscript,mathrsfs}
\usepackage{color}
\usepackage{dsfont}
\usepackage[left=2cm,right=2cm,top=3.5cm,bottom=3.5cm]{geometry}
\usepackage{color}
\usepackage[framemethod=tikz]{mdframed}
\allowdisplaybreaks

\usepackage{soul}

\catcode`\@=11 \@addtoreset{equation}{section}

\catcode`\@=12

\newtheorem{Theorem}{Theorem}[section]
\newtheorem{Proposition}[Theorem]{Proposition}
\newtheorem{Lemma}[Theorem]{Lemma}
\newtheorem{Corollary}[Theorem]{Corollary}

\theoremstyle{definition}
\newtheorem{Definition}[Theorem]{Definition}

\newtheorem{Remark}[Theorem]{Remark}

\newcommand{\bTheorem}[1]{
	\begin{Theorem} \label{T#1} }
	\newcommand{\eT}{\end{Theorem}}

\newcommand{\bProposition}[1]{
	\begin{Proposition} \label{P#1}}
	\newcommand{\eP}{\end{Proposition}}

\newcommand{\bLemma}[1]{
	\begin{Lemma} \label{L#1} }
	\newcommand{\eL}{\end{Lemma}}

\newcommand{\bCorollary}[1]{
	\begin{Corollary} \label{C#1} }
	\newcommand{\eC}{\end{Corollary}}

\newcommand{\bRemark}[1]{
	\begin{Remark} \label{R#1} }
	\newcommand{\eR}{\end{Remark}}

\newcommand{\bDefinition}[1]{
	\begin{Definition} \label{D#1} }
	\newcommand{\eD}{\end{Definition}}

\newcommand{\Del}{\Delta_x}

\newcommand{\Ds}{\mathbb{D}_x}

\newcommand{\bFormula}[1]{
	\begin{equation} \label{#1}}
	\newcommand{\eF}{\end{equation}}

\newcommand{\vr}{\varrho}

\newcommand{\vt}{\vartheta}
\newcommand{\vu}{\vc{u}}

\newcommand{\vc}[1]{{\bf #1}}

\newcommand{\Div}{{\rm div}_x}
\newcommand{\Grad}{\nabla_x}

\newcommand{\dx}{\,{\rm d} {x}}

\newcommand{\dt}{\,{\rm d} t }

\newcommand{\vU}{\vc{U}}

\newcommand{\D}{{\rm d}}

\newcommand{\expe}[1]{ \mathbb{E} \left[ #1 \right] }

\newcommand{\br}{ \nonumber \\ }

\def\softd{{\leavevmode\setbox1=\hbox{d}%
		\hbox to 1.05\wd1{d\kern-0.4ex{\char039}\hss}}}
\definecolor{Cgrey}{rgb}{0.85,0.85,0.85}
\definecolor{Cblue}{rgb}{0.50,0.85,0.85}
\definecolor{Cred}{rgb}{1,0,0}
\definecolor{fancy}{rgb}{0.10,0.85,0.10}

\newcommand\Cbox[2]{%
	\newbox\contentbox%
	\newbox\bkgdbox%
	\setbox\contentbox\hbox to \hsize{%
		\vtop{
			\kern\columnsep
			\hbox to \hsize{%
				\kern\columnsep%
				\advance\hsize by -2\columnsep%
				\setlength{\textwidth}{\hsize}%
				\vbox{
					\parskip=\baselineskip
					\parindent=0bp
					#2
				}%
				\kern\columnsep%
			}%
			\kern\columnsep%
		}%
	}%
	\setbox\bkgdbox\vbox{
		\color{#1}
		\hrule width  \wd\contentbox %
		height \ht\contentbox %
		depth  \dp\contentbox
		\color{black}
	}%
	\wd\bkgdbox=0bp%
	\vbox{\hbox to \hsize{\box\bkgdbox\box\contentbox}}%
	\vskip\baselineskip%
}

\mdfdefinestyle{MyFrame}{%
	linecolor=black,
	outerlinewidth=1pt,
	roundcorner=5pt,
	innertopmargin=\baselineskip,
	innerbottommargin=\baselineskip,
	innerrightmargin=10pt,
	innerleftmargin=10pt,
	backgroundcolor=white!20!white}

\begin{document}


\title{Statistical solutions for the Navier--Stokes--Fourier system}

\author{Eduard Feireisl
	\thanks{The work of E.F. was partially supported by the
		Czech Sciences Foundation (GA\v CR), Grant Agreement
		21--02411S. The Institute of Mathematics of the Academy of Sciences of
		the Czech Republic is supported by RVO:67985840. \newline
		\hspace*{1em} $^\spadesuit$
		M.L. has been funded by the Deutsche Forschungsgemeinschaft (DFG, German Research Foundation) - Project number 233630050 - TRR 146. She is grateful to the Gutenberg Research College
		and Mainz Institute of Multiscale Modelling for supporting her research.
	} \and  M\'aria Luk\'a\v{c}ov\'a-Medvi\softd ov\'a$^{\spadesuit}$
}

\date{\today}

\maketitle

\bigskip

\centerline{$^*$  Institute of Mathematics of the Academy of Sciences of the Czech Republic}

\centerline{\v Zitn\' a 25, CZ-115 67 Praha 1, Czech Republic}

\medskip

\centerline{$^\spadesuit$ Institute of Mathematics, Johannes Gutenberg--University Mainz}

\centerline{Staudingerweg 9, 55 128 Mainz, Germany}

\begin{abstract}
	
	We show a general stability result in the framework of strong solutions of the Navier--Stokes--Fourier system describing the motion of a compressible viscous and heat conducting gas. As a corollary, we develop a concept
	of statistical solution in the class of regular solutions ``beyond  the blow up time''.

\end{abstract}



\tableofcontents

\section{Introduction}
\label{i}

The goal of the present paper is to develop a mathematical framework for studying the motion of a general compressible, viscous and heat conducting fluid with random (uncertain) data. Basically all well established models
in continuum fluid mechanics, notably the Navier--Stokes system, are known to be well posed in the natural framework of smooth or at least sufficiently regular solutions. Unfortunately, the solutions may develop a finite time blow up for certain data, see Merle et al. \cite{MeRaRoSz}, \cite{MeRaRoSzbis}, whereas the problem is expected to be ill posed in the class of generalized (weak) solutions, cf. Buckmaster and Vicol \cite{BucVic}
or Albritton et al. \cite{AlbBruCol}.

Unique dependence of solutions on the data seems mandatory in problems with random parameters. If at least existence of some kind of generalized solution is known, uniqueness can be restored by means a suitable \emph{selection} process, see \cite{FanFei}. In this case, the associated statistical solutions can be identified
with a Markov semigroup of operators defined on a suitable space of probability distributions. For a weaker
concept of statistical solution based on the weak solutions, we refer to Constantin and Wu \cite{ConWu}, Foias \cite{Foias}, Foias, Rosa, and Temam \cite{FoRoTe3}, \cite{FoRoTe1}, Vishik and Fursikov \cite{VisFur}.
For compressible Euler system statistical solutions were identified with a hierarchie of correlation measures by Fjordholm et al. \cite{FJ_statistic}.
Unfortunately, for certain class of problems, the existence of a suitable generalized solution is still open. In particular, solutions may loose regularity or even become unbounded in a finite time lap, which excludes any ``natural'' continuation.

Statistical solutions are inevitable in the analysis of uncertainty quantification methods, such as the collocation, Monte Carlo or stochastic Galerkin methods. These methods are routinely used in engineering, physics, meteorology or medicine to quantify the uncertainty of data.
We refer to monographs and review papers by Abgrall and Mishra \cite{AM17}, Le Ma{\^\i}tre and Knio \cite{LK10}, Mishra and Schwab \cite{Schwab}, Nordstr\"om \cite{Nordstroem2015}, Xiu \cite{XD}. Rigorous convergence analysis of these uncertainty quantification methods
leans on uniqueness
and, ideally, continuous dependence of solutions on random parameters and was done only for simplified model problems.  For complex system, such as the compressible Navier--Stokes--Fourier equations rigourous
analysis of uncertainty quantification method is not available. Our aim in this paper is to fill this gap and provide a rigorous solution concept for uncertainty quantification analysis.

The \emph{Navier--Stokes--Fourier} (NSF) \emph{system} describes the time
evolution of the mass density $\vr = \vr(t,x)$, the (absolute) temperature $\vt = \vt(t,x)$, and velocity
$\vu = \vu(t,x)$ of a general Newtonian compressible heat conducting fluid:

\begin{mdframed}[style=MyFrame]

	\begin{itemize}
		\item{ {\bf Mass conservation, equation of continuity:}}
\begin{equation}		
		\partial_t \vr + \Div (\vr \vu) = 0. \label{nsf1}
		\end{equation}
	
	\item{{\bf Momentum balance, Newton's second law:}}
	
	 \begin{equation}
		\partial_t (\vr \vu) + \Div (\vr \vu \otimes \vu) + \Grad p(\vr, \vt) =
		\Div \mathbb{S}(\Ds \vu) + \vr \vc{g}. \label{nsf2}
		\end{equation}
	
	\item{{\bf Internal energy balance, First law of thermodynamics:}}
	
	\begin{equation}
		 \partial_t (\vr e(\vr, \vt)) + \Div (\vr e(\vr, \vt) \vu) + \Div \vc{q} = \mathbb{S}(\Ds \vu): \Ds \vu - p(\vr, \vt) \Div \vu + \vr Q.
		\label{nsf3}
		\end{equation}
			\end{itemize}

	\end{mdframed}

\noindent
For viscous and heat conducting gases, the problem may be formally closed by the following \emph{constitutive
relations}:

\begin{mdframed}[style=MyFrame]

	\begin{itemize}
		
		\item {\bf Equations of state, Boyle--Mariotte law:}
	
	\begin{equation}
		p(\vr, \vt) = \vr \vt,\ e(\vr, \vt) = c_v \vt . \label{nsf4}
		\end{equation}
	
	\item {\bf Newton's rheological law:}
	
	\begin{equation}
		 	\mathbb{S}(\Ds \vu) = \mu \left(\Grad \vu + \Grad \vu^t - \frac{2}{3} \Div \vu \mathbb{I} \right) +
		\eta \Div \vu \mathbb{I}, \  \   \mu > 0, \eta \geq 0, \label{nsf5}
		\end{equation}
	
	\item {\bf Fourier's law:}
	
	\begin{equation}
		\vc{q} = - \kappa \Grad \vt, \ \kappa  > 0. \label{nsf6}
		\end{equation}
		
	\end{itemize}

	\end{mdframed}

\noindent The initial state of the system at the reference time $t = 0$ is given through
\emph{initial data}:

\begin{mdframed}[style=MyFrame]

\begin{equation} \label{nsf7}
	\vr(0,\cdot) = \vr_0,\ \vu(0, \cdot) = \vu_0,\ \vt(0, \cdot) = \vt_0.
	\end{equation}
	
	\end{mdframed}

\noindent For simplicity, we assume that all quantities are spatially periodic, meaning the fluid domain can be identified with the flat torus
\begin{equation} \label{BC}
	\mathbb{T}^3 = \left( \left[-1,1 \right] \Big|_{\{ -1;1 \}} \right)^3.
	\end{equation}

The quantity
\[
D = \Big( \vr_0, \vu_0, \vt_0; \vc{g}, Q; c_v, \mu, \eta, \kappa \Big)
\]
represents the \emph{data} of the problem. Ideally, for any \emph{physically admissible} data and the time
$t > 0$, the NSF system should admit a unique solution $[\vr(t, \cdot), \vt(t, \cdot), \vu(t, \cdot)]$.
By physically \emph{admissible} we mean the data satisfying
\begin{equation} \label{admis}
\inf \vr_0 > 0,\ \inf \vt_0 > 0, \ c_v > 1,\ \mu > 0,\ \eta \geq 0, \ \kappa > 0.
\end{equation}
Hereafter we always tacitly assume that the data we deal with are admissible.

\subsection{Local existence of regular solutions}
\label{LE}

Probably the
``optimal'' existence result in the class of strong solutions was established by
Cho and Kim \cite[Theorem 1]{ChoKim1}.

\begin{Theorem}[{\bf Local existence}] \label{LET1}
	Let $3 < q \leq 6$ be given. Let
	\begin{equation} \label{LE1}
		\vr_0 \in W^{1,q}(\mathbb{T}^3),\ \vr_0 > 0,\
		\vt_0 \in W^{2,2}(\mathbb{T}^3),\ \vt_0 > 0, \
		\vu_0 \in W^{2,2}(\mathbb{T}^3; R^3),
		\end{equation}
and
\begin{align}
	\vc{g} &\in BC([0,\infty); L^2(\mathbb{T}^3; R^3))
	\cap L^2(0, \infty; L^q(\mathbb{T}^3; R^3)),\
	\partial_t \vc{g} \in L^2(0, \infty; W^{-1,2}(\mathbb{T}^3; R^3)), \br
	Q &\in BC([0,\infty); L^2(\mathbb{T}^3))
	\cap L^2(0, \infty; L^q(\mathbb{T}^3)),\
	\partial_t Q \in L^2(0, \infty; W^{-1,2}(\mathbb{T}^3)), \ Q \geq 0.
	\label{LE2}
	\end{align}

Then there exists $T > 0$ and a strong solution of the Navier--Stokes--Fourier system \eqref{nsf1}--\eqref{nsf7}
unique in the class	
\begin{align}
	\vr &\in C([0,T]; W^{1,q}(\mathbb{T}^3)),\ \partial_t \vr \in C([0,T]; L^q(\mathbb{T}^3)), \ \vr > 0, \label{LE3} \\
		\vt &\in C([0,T]; W^{2,2}(\mathbb{T}^3)) \cap L^2(0,T; W^{2,q}(\mathbb{T}^3)),\ \vt > 0, \br
		\partial_t \vt &\in L^\infty(0,T; L^2(\mathbb{T}^3)) \cap L^2(0,T; W^{1,2}(\mathbb{T}^3)), \label{LE4}\\
		\vu &\in C([0,T]; W^{2,2}(\mathbb{T}^3; R^3)) \cap L^2(0,T; W^{2,q}(\mathbb{T}^3;R^3)), \br
			\partial_t \vu &\in L^\infty(0,T; L^2(\mathbb{T}^3; R^3)) \cap L^2(0,T; W^{1,2}(\mathbb{T}^3; R^3)). \label{LE5}	
		\end{align}

	\end{Theorem}

\begin{Remark} \label{RLE1}
	
	A brief inspection of the existence proof in \cite{ChoKim1} reveals that the
	length of the existence interval $T_{\rm max}$ can be bounded below by a positive constant depending solely on the
	norm of the initial data and forcing terms in the function spaces specified in \eqref{LE1} and \eqref{LE2} as well as on $\inf_{\mathbb{T}^3} \vr_0 > 0$.
	In particular, $T_{\rm max}$ is  independent of $\inf \vt_0$. As a matter of fact, the local existence established in \cite{ChoKim1} holds for non--negative $\vr_0$, $\vt_0$ under suitable compatibility
	conditions.

	\end{Remark}

We define the life span -- maximal existence time $T_{\rm max}$ -- of the local solution,
\begin{equation} \label{LE6}
0< 	T_{\rm max} = \sup \left\{ T > 0 \ \Big|\ \mbox{regular solution}\  (\vr, \vt, \vu) \ \mbox{exists in}\ [0,T] \right\} \leq \infty.
	\end{equation}
In accordance with Remark \ref{RLE1},
\begin{align} \label{LE7}
	T_{\rm max} &< \infty \br &\Rightarrow \br
	\lim_{t \to T_{\rm max} -} &\left( \| \vr(t, \cdot) \|_{W^{1,q}(\mathbb{T}^3)} + \| \vt(t, \cdot) \|_{W^{2,2}(\mathbb{T}^3)} + \| \vu(t, \cdot) \|_{W^{2,2}(\mathbb{T}^3; R^3)} + \| \vr(t, \cdot)^{-1} \|_{C(\mathbb{T}^3)} \right)
 = \infty .
	\end{align}
Note that $	\| \vr^{-1} \|_{C(\mathbb{T}^3)} = \inf_{\mathbb{T}^3} \vr ^{-1} .$

\subsection{Conditional regularity}

It is still open whether $T_{\rm max}$ in Theorem \ref{LET1} may be finite for some data. If this is the case, however, the density and/or temperature must become unbounded at the same time. We report the following \emph{conditional regularity} criterion, see
\cite[Theorem 1.4]{FeWeZh}:

	\begin{Theorem}[{\bf Conditional regularity criterion}] \label{BCT1}
		
		Let the data
		\[
		c_v > 1,\ \mu > 0,\ \eta \geq 0, \kappa > 0
		\]
		be given.
		Under the hypotheses of Theorem \ref{LET1}, let $[\vr, \vt, \vu]$ be the local strong solution of the NSF system defined on $[0, T_{\rm max})$.
		
		Then there exists a function $\Lambda$, bounded for bounded values of its argument, such that
		\begin{align} \label{BC1}
			\sup_{t \in [0,T]} &\| \vr(t, \cdot) \|_{W^{1,q}(\mathbb{T}^3)}	+
			\sup_{t \in [0,T]} \| \vt(t, \cdot) \|_{W^{2,2}(\mathbb{T}^3)} +
			\sup_{t \in [0,T]} \| \vu(t, \cdot) \|_{W^{2,2}(\mathbb{T}^3; R^3)} \br
			&+ \int_0^T \|  \vt \|_{W^{2,q}(\mathbb{T}^3)}^2 \dt + \int_0^T \|  \vu \|_{W^{2,q}(\mathbb{T}^3;R^3)}^2 \dt \br
			&+ \sup_{t \in [0,T]} \| \partial_t \vr \|_{L^q(\mathbb{T}^3; R^3)} + \sup_{t \in [0,T]} \| \partial_t \vt \|_{L^2(\mathbb{T}^3)}   + \sup_{t \in [0,T]} \| \partial_t \vu \|_{L^2(\mathbb{T}^3;R^3)}\br
			&+ \int_0^T \| \partial_t \vt \|_{W^{1,2}(\mathbb{T}^3)}^2 \dt + \int_0^T \| \partial_t \vu \|_{W^{1,2}(\mathbb{T}^3; R^3)}^2 \dt \br
			&\leq \Lambda \Big( T , \| \vr_0 \|_{W^{1,q}(\mathbb{T}^3)} , \| \vt_0 \|_{W^{2,2}(\mathbb{T}^3)}
			, \| \vu_0 \|_{W^{2,2}(\mathbb{T}^3; R^3)}, \| \vr_0^{-1} \|_{C(\mathbb{T}^3)},
				\br
			&\quad \sup_{t \in [0,T]} \| \vc{g}(t, \cdot) \|_{L^2(\mathbb{T}^3; R^3 )}, \sup_{t \in [0,T]} \| Q(t, \cdot) \|_{L^2(\mathbb{T}^3 )}, \int_0^T \| \vc{g} \|^2_{L^q(\mathbb{T}^3; R^3))} \dt,
			\int_0^T \| Q \|^2_{L^q(\mathbb{T}^3))} \dt,  \br
			&\quad  \left. \int_0^T \| \partial_t \vc{g} \|^2_{W^{-1,2}(\mathbb{T}^3; R^3))} \dt,
			\int_0^T \| \partial_t Q \|^2_{W^{-1,2}(\mathbb{T}^3))} \dt,\ \sup_{(0,T) \times \mathbb{T}^3} \vr ,
			\sup_{(0,T) \times \mathbb{T}^3} \vt, c_v, \mu, \mu^{-1}, \eta, \kappa, \kappa^{-1} \right)
		\end{align}
		for any $0 < T < T_{\rm max}$.	
	\end{Theorem}

Besides the given data, the function $\Lambda$ depends only on the uniform norm of $\vr$ and $\vt$.
The conclusion of Theorem \ref{BCT1} can be therefore equivalently formulated as
\begin{equation} \label{BC2}
	\limsup_{t \to T_{\rm max}-} \ \sup_{x \in \mathbb{T}^3} \Big( \vr(t, x) + \vt(t, x) \Big) = \infty.
	\end{equation}

\subsection{Data dependence and statistical solutions}

In the absence of global--in--time existence results in the class of regular solutions, the weak solutions may represent a suitable alternative,
Unfortunately, the existence theory developed in \cite{FeiNovOpen} is restricted to
a rather specific class of constitutive relations that does not cover \eqref{nsf4}--\eqref{nsf6}. What is more,
the available {\it a priori} bounds are not sufficient to use the framework of the more general measure--valued solutions in the spirit of \cite{BreFeiNov}. In particular, there is no guarantee the solution $(\vr, \vt, \vu)$
remains continuous up to the blow up time $T_{\rm max}$ even in a weak sense. Indeed the only available
{\it a priori} bounds result from the physical principles of \emph{mass conservation}
\begin{equation} \label{conM}
	\int_{\mathbb{T}^3} \vr (t, \cdot) \ \dx = \int_{\mathbb{T}^3} \vr_0 (t, \cdot) \ \dx,
	\end{equation}
\emph{energy conservation} (if $\vc{g}= 0,\, Q = 0$)
\begin{equation} \label{conE}
	\int_{\mathbb{T}^3} \left( \frac{1}{2} \vr |\vu|^2 + c_v \vr \vt \right) (t, \cdot) \ \dx =
	\int_{\mathbb{T}^3} \left( \frac{1}{2} \vr_0 |\vu_0|^2 + c_v \vr_0 \vt_0 \right)  \ \dx,	
\end{equation}
and \emph{entropy production}
\begin{equation} \label{entrP}
	\int_{\mathbb{T}^3} \vr \log \left( \frac{\vt^{c_v}}{\vr} \right) (t, \cdot) \dx \geq
	\int_{\mathbb{T}^3} \vr_0 \log \left( \frac{\vt^{c_v}_0}{\vr_0} \right) \dx +
	\int_0^t \int_{\mathbb{T}^3} \frac{1}{\vt} \left( \mathbb{S} (\Ds \vu) : \Ds \vu + \kappa \frac{ |\Grad \vt |^2}{\vt}        \right) \dx \ \D s  .		
	\end{equation}

Our plan is to develop the theory for random data, notably the concept of \emph{statistical} solution, using the framework of
regular solutions. In view of the blow up criterion \eqref{BC2}, this approach suits well to problems related
to convergence of numerical approximations, where \eqref{BC2} is very often taken for granted. In addition,
we may also ``anticipate'' the idea that solutions violating \eqref{BC2} and thus blowing up in a finite time
are statistically insignificant and can be ignored in large samples of data.

The theory is based on a stability result that may be of independent interest: The life span $T_{\rm max}$ is a lower semi--continuous function of with respect to a suitable topology of the data space, see Section \ref{TM}. Given a regular solution of the NSF system, any other solution corresponding to a sufficiently close data remains regular at least on the same time interval. Similar results in the context of barotropic fluids were obtained by Bae and Zajaczkowski
\cite{BaeZaj}, see also
\cite{BeFeJiNo}.
Once stability is established, the statistical solutions are obtained via the push forward measure argument, see Section \ref{SE}. In order to accommodate the possible blow up of solutions, we modify the original topology of the phase space
in the spirit of the celebrated Jakubowski Theorem \cite{Jakub}. Finally, in Section~\ref{MC} we illustrate the theory showing convergence of the Monte-Carlo method.

\section{Stability}
\label{TM}

We denote $(\vr, \vt, \vu)[D]$ the strong solution of the NSF system corresponding to the data
\[
D = \left( \vr_0, \vt_0, \vu_0; \vc{g}, Q; c_v, \mu, \eta, \kappa \right).
\]
We introduce the data space $D_{X,F,P}$,
\[
D_{X,F,P} = X \times F \times P,
\]
where $X \owns (\vr_0, \vt_0, \vu_0)$ is the \emph{phase space}, $F \owns (\vc{g}, Q) $ the space of
\emph{external forces}, and $P \owns (c_v, \mu, \eta, \kappa)$ the \emph{parameter space}.
Furthermore,
motivated by the local existence result stated in Theorem \ref{LET1}, the afore mentioned spaces are endowed with the following topologies:

\begin{itemize}
	\item
The
\emph{phase space}
\[
X =  W^{1,q}(\mathbb{T}^3) \times
W^{2,2}(\mathbb{T}^3) \times W^{2,2}(\mathbb{T}^3; R^3)
\]
-- a separable reflexive Banach space compactly embedded into the space of continuous functions
\[
 C(\mathbb{T}^3; R^5).
\]

\item The \emph{trajectory space}
\begin{align}
X_T &= C([0,T]; W^{1,q}(\mathbb{T}^3)) \cap C^1([0,T]; L^{q}(\mathbb{T}^3)) \br
&\times  C([0,T]; W^{2,2}(\mathbb{T}^3)) \cap L^2(0,T; W^{2,q}(\mathbb{T}^3))
\cap W^{1,\infty}(0,T; L^2 (\mathbb{T}^3)) \cap W^{1,2}(0,T; W^{1,2}(\mathbb{T}^3)) \br
&\times  C([0,T]; W^{2,2}(\mathbb{T}^3;R^3)) \cap L^2(0,T; W^{2,q}(\mathbb{T}^3;R^3)) \br
&\quad \quad \quad \quad \cap W^{1,\infty}(0,T; L^2 (\mathbb{T}^3;R^3)) \cap W^{1,2}(0,T; W^{1,2}(\mathbb{T}^3;R^3))
\nonumber
\end{align}
-- a separable Banach space compactly embedded into the space of continuous functions
\[
C([0,T] \times \mathbb{T}^3; R^5).
\]

\item The \emph{data space}
\[
D_{X,F,P} = X \times F \times P,
\]
where
\begin{align}
F &= \left[  BC[0, \infty; L^2 (\mathbb{T}^3; R^3) \cap L^2(0,\infty; L^q (\mathbb{T}^3; R^3) \cap
W^{1,2}(0, \infty; W^{-1,2}(\mathbb{T}^3; R^3) ) \right] \br
&\times
\left[  BC[0, \infty; L^2 (\mathbb{T}^3) \cap L^2(0,\infty; L^q (\mathbb{T}^3) \cap
W^{1,2}(0, \infty; W^{-1,2}(\mathbb{T}^3) ) \right], \nonumber
\end{align}
\[
P = R^4.
\]

\end{itemize}

\subsection{Stability in the basic topology}

In accordance with Theorem \ref{LET1}, the solution exists on a maximal time interval $[0, T_{\rm max}[D])$ as long as
the data $D$ are admissible and belong to $D_{X,F,P}$. Motivated by the regularity criterion established in Theorem \ref{BCT1}, we introduce  a ``stopping time''
\begin{equation} \label{TM1}
	{T}_M [D] = \sup_{\tau \in [0, T_{\rm max})} \left\{ \sup_{t \in [0,\tau], x \in \mathbb{T}^3} \Big( \vr(t, x) +
	\vt(t, x) \Big) < M \right\} = \inf_{\tau > 0} \sup_{x \in \mathbb{T}^3} \Big( \vr(\tau, x) +
	\vt(\tau, x) \Big) \geq M.
	\end{equation}
Obviously, ${T}_M[ D ]$ is a  non--decreasing functions of $M$, and, in view of Theorem \ref{BCT1} and \eqref{BC2},
\begin{equation} \label{TM2}
{T}_M[D] < T_{\rm max}[D],\ {T}_M [D] \nearrow T_{\rm max}[D] \ \mbox{as}\ M \to \infty.
\end{equation}
	
Now, suppose that $D_n \to D$ in $D_{X,F,P}$, where $D$ are admissible data, specifically,
\begin{align}
	\vr_{0,n} &\to \vr_0 \ \mbox{in} \ W^{1,q}(\mathbb{T}^3), \ \mbox{where}\  \vr_0 > 0, \br
	\vt_{0,n} &\to \vt_0 \ \mbox{in}\ W^{2,2}(\mathbb{T}^3), \ \vt_0 > 0 , \br
	\vu_{0,n} &\to \vu_0 \ \mbox{in}\ W^{2,2}(\mathbb{T}^3; R^3),
	\label{TM3}	
\end{align}
\begin{align}
	\vc{g}_n &\to \vc{g} \ \mbox{in}\ BC([0,\infty; L^2(\mathbb{T}^3; R^3)) \cap
	L^2(0,\infty; L^q(\mathbb{T}^3; R^3)), \ \partial_t \vc{g}_n \to \partial_t \vc{g}\
	\mbox{in}\ L^2(0,\infty; W^{-1,2}(\mathbb{T}^3; R^3)), \br
	Q_n &\to Q \ \mbox{in}\ BC([0,\infty); L^2(\mathbb{T}^3)) \cap
	L^2(0,\infty; L^q(\mathbb{T}^3)), \ \partial_t Q_n \to \partial_t Q\
	\mbox{in}\ L^2(0,\infty; W^{-1,2}(\mathbb{T}^3)),
	\label{TM4}
\end{align}
and
\begin{equation} \label{TM5}
	c^n_v \to c_v > 1,\ \mu_n \to \mu > 0,\ \lambda_n \to \lambda \geq 0,\ \kappa_n \to \kappa > 0.
	\end{equation}
In addition, assume that
\[
{T}_M[ D_n] \to \tau \ \mbox{and fix}\
0 \leq T^- \leq \tau < T^+.
\]
Finally, consider the sequence $(\vr_n, \vt_n, \vu_n) = (\vr, \vt, \vu)[D_n](t \wedge {T}_M[D_n], \cdot)$, specifically
$(\vr_n, \vt_n, \vu_n)(t, \cdot)$ is the strong solution of the NSF system if $t \leq {T}_M(D_n)$ and
$(\vr_n, \vt_n \vu_n)(t, \cdot) = (\vr_n, \vt_n \vu_n)({T}_M(D_n), \cdot)$ if $t > {T}_M[ D_n] $.

Now, it follows from the uniform bounds \eqref{BC1} and the standard compact embedding relations
for Sobolev spaces that
\[
( \vr_n, \vt_n, \vu_n )_{n=1}^\infty \ \mbox{if precompact in}\ BC([0, \infty) \times \mathbb{T}^3); R^5),
\]
and, by the same token
\[
\vr_n \to \vr,\ \vt_n \to \vt,\ \vu_n \to \vu
\]
in the weak-(*) topology of the trajectory space $X_{T^+}$. Here and hereafter, by the weak-(*) topology on the trajectory space $X_T$ we mean the weak star topology of the space
\begin{align}
	X_{T,w-(*)} &= L^\infty(0,T; W^{1,q}(\mathbb{T}^3)) \cap W^{1,\infty}(0,T; L^{q}(\mathbb{T}^3)) \br
	&\times  L^\infty(0,T; W^{2,2}(\mathbb{T}^3)) \cap L^2(0,T; W^{2,q}(\mathbb{T}^3))
	\cap W^{1,\infty}(0,T; L^2 (\mathbb{T}^3)) \cap W^{1,2}(0,T; W^{1,2}(\mathbb{T}^3)) \br
	&\times  L^\infty(0,T; W^{2,2}(\mathbb{T}^3;R^3)) \cap L^2(0,T; W^{2,q}(\mathbb{T}^3;R^3)) \br
	&\quad \quad \quad \quad \cap W^{1,\infty}(0,T; L^2 (\mathbb{T}^3;R^3)) \cap W^{1,2}(0,T; W^{1,2}(\mathbb{T}^3;R^3)).
	\nonumber
\end{align}
The following is easy to check:
\begin{itemize}
	\item The limit $(\vr, \vt, \vu)$ is the unique solution of the NSF system in $[0, T^-]$ corresponding to the limit data $D$.
	\item
	\[
	\sup_{x \in \mathbb{T}^3} \Big( \vr(t, \cdot) + \vt(t \cdot) \Big) = M \ \mbox{for any}\
	t \geq T^+.
	\]
	\end{itemize}
As $T^- \leq \tau \leq T^+$ are arbitrary, we conclude ${T}_M[D] \leq \tau$, in other words, the mapping
\begin{equation} \label{LSC}
D \mapsto {T}_M[D] \ \mbox{is lower semi--continuous}.
\end{equation}
Accordingly, $T_{\rm max}$, being a supremum of l.s.c. mappings, is lower semi--continuous function of $D$. We have shown the following result.

\begin{mdframed}[style=MyFrame]
	
	\begin{Theorem}[{\bf Stability}] \label{TST1}
		
	 For any sequence $D_n \to D$ in $D_{X,F,P}$, where $D$ are admissible data, we have
	 \begin{equation} \label{LSC1}
	 \liminf_{n \to \infty} T_{\rm max}[D_n] \geq T_{\rm max}[D],
	 \end{equation}
	 and
	 \begin{equation} \label{conver}
	 (\vr, \vt, \vu)[D_n] \to (\vr, \vt, \vu)[D] \ \mbox{weakly-(*) in}\ X_{T}
	 \end{equation}
	 for any $0 < T < T_{\rm max}[D]$.

		\end{Theorem}

	\end{mdframed}

\begin{Remark} \label{TSR1}
	
	As agreed,
	the convergence \eqref{conver} means
	\begin{align}
		\vr_n &\to \vr \ \mbox{weakly-(*) in}\ L^\infty(0,T; W^{1,q}(\mathbb{T}^3)), \br
		\vt_n &\to \vt \ \mbox{weakly-(*) in}\ L^\infty(0,T; W^{2,2}(\mathbb{T}^3))
		\ \mbox{and weakly in}\ L^2(0,T; W^{2,q}(\mathbb{T}^3)), \br
		\vu_n &\to \vu \ \mbox{weakly-(*) in}\ L^\infty(0,T; W^{2,2}(\mathbb{T}^3; R^3))
		\ \mbox{and weakly in}\ L^2(0,T; W^{2,q}(\mathbb{T}^3; R^3)), \br
		\partial_t \vr_n &\to \partial_t \vr \ \mbox{weakly-(*) in}\ L^\infty(0,T; L^{q}(\mathbb{T}^3)),\br
		\partial_t \vt_n &\to \partial_t \vt \ \mbox{weakly-(*) in}\ L^\infty(0,T; L^{2}(\mathbb{T}^3))
		\ \mbox{and weakly in}\ L^2(0,T; W^{1,2}(\mathbb{T}^3)), \br
		\partial_t \vu_n &\to \partial_t \vu \ \mbox{weakly-(*) in}\ L^\infty(0,T; L^{2}(\mathbb{T}^3; R^3))
		\ \mbox{and weakly in}\ L^2(0,T; W^{1,2}(\mathbb{T}^3; R^3)),
		\label{conver1}
		\end{align}
	in particular,
	\[
[\vr, \vt, \vu](D_n) \to [\vr, \vt, \vu](D) \ \mbox{in}\ C([0,T] \times \mathbb{T}^3; R^5).
\]	
	\end{Remark}

\subsection{Higher regularity of solutions}

In the section, we consider the data enjoying higher regularity properties.

\subsubsection{Local existence by Valli and Zajaczkowski}

Local existence result can be shown in higher regularity spaces -- see \cite[Theorem A]{Vall1}, \cite{Vall2} and also
Valli and Zajaczkowski \cite{VAZA}.
\begin{Theorem}[{\bf Local existence [Valli]}] \label{LET2}
	Let
	\begin{equation} \label{LE8}
		\vr_0 \in W^{3,2}(\mathbb{T}^3),\ \vr_0 > 0,\
		\vt_0 \in W^{3,2}(\mathbb{T}^3),\ \vt_0 > 0, \
		\vu_0 \in W^{3,2}(\mathbb{T}^3; R^3),
	\end{equation}
	and
	\begin{align}
		\vc{g} &\in  L^2(0, \infty; W^{2,2}(\mathbb{T}^3; R^3)),\
		\partial_t \vc{g} \in L^2(0, \infty; L^2(\mathbb{T}^3; R^3)), \br
		Q &\in L^2(0, \infty; W^{2,2} (\mathbb{T}^3)),\
		\partial_t Q \in L^2(0, T; L^2(\mathbb{T}^3)), \ Q \geq 0.
		\label{LE9}
	\end{align}
	
	Then there exists $T > 0$ and a strong solution of the Navier--Stokes--Fourier system \eqref{nsf1}--\eqref{nsf7}
	unique in the class	
	\begin{align}
		\vr &\in C([0,T]; W^{3,2}(\mathbb{T}^3)),\ \partial_t \vr \in C([0,T]; W^{2,2}(\mathbb{T}^3)), \ \vr > 0, \label{LE10} \\
		\vt &\in C([0,T]; W^{3,2}(\mathbb{T}^3)) \cap L^2(0,T; W^{4,2}(\mathbb{T}^3)),\ \vt > 0, \br
		\partial^2_t \vt &\in L^2(0,T; L^{2}(\mathbb{T}^3)), \label{LE11}\\
		\vu &\in C([0,T]; W^{3,2}(\mathbb{T}^3; R^3)) \cap L^2(0,T; W^{4,2}(\mathbb{T}^3;R^3)), \br
		\partial^2_t \vu &\in L^2(0,T; L^2 (\mathbb{T}^3; R^3)). \label{LE12}	
	\end{align}

\end{Theorem}

Similarly \eqref{LE6}, \eqref{LE7}, we can define the maximal time interval $T^{V}_{\rm max}$ on which the solution
exists. Obviously,
\[
T_{\rm max}^{V} \leq T_{\rm max},
\]
where $T_{\rm max}$ is the life--span introduced in \eqref{LE7}.

\begin{Lemma} \label{LEL1}
	Under the hypotheses of Theorem \ref{LET1},
	\[
	T_{\rm max} = T^{V}_{\rm max}
	\]
	
\end{Lemma}

\begin{proof}
	
	In view of the regularity criterion \cite[Theorem 2.1]{FeNoSun1}, it is enough to show that any solution
	belonging to the regularity class \eqref{LE3}--\eqref{LE5} satisfies
	\begin{equation} \label{LE13}
		\Grad \vu \in L^\infty((0,T) \times \mathbb{T}^3; R^{3 \times 3}).
	\end{equation}
	The momentum equation reads
	\[
	\partial_t \vu - \frac{1}{\vr} \Div \mathbb{S}(\Ds \vu) = - \vu \cdot \Grad \vu - \frac{1}{\vr}
	\Grad p(\vr, \vt) + \vc{g}.
	\]	
	As $(\vr, \vt, \vu)$ belong to the class \eqref{LE3}--\eqref{LE5}, it is easy to check that
	\[
	- \vu \cdot \Grad \vu - \frac{1}{\vr}
	\Grad p(\vr, \vt) + \vc{g} \in L^\infty(0,T; L^q (\mathbb{T}^d; R^3)),
	\]	
	where $3 < q \leq 6$ is the exponent introduced in Theorem \ref{LET1}. Note that, in view of \eqref{LE5},
	$\vr$ is bounded below in terms of $\inf \vr_0$, $T$ and $\| \Div \vu \|_{L^1(0,T; L^\infty(\mathbb{T}^3))}$.
	Consequently, in view of the maximal $L^p-L^q$ regularity estimates, we get
	\[
	\partial_t \vu \in L^p(0,T; L^q(\mathbb{T}^3; R^3)),\
	\vu \in L^p (0,T; W^{2,q} (\mathbb{T}; R^3) \ \mbox{for any}\ 1 \leq p < \infty
	\]
	as soon as $\vu_0 \in W^{2,q}(\mathbb{T}^3; R^3) \subset W^{3,2}(\mathbb{T}^3; R^3)$. Consequently, for any
	$\alpha < 1$, there exists $p > 1$ large enough so that
	\[
	\vu \in C([0,T]; W^{2 \alpha ,q}(\mathbb{T}^3; R^3)) \ \Rightarrow\
	\Grad \vu \in C([0,T]; W^{2\alpha - 1 ,q}(\mathbb{T}^3; R^{3 \times 3} )).
	\]
	Thus choosing $(2 \alpha - 1)q > 3$ yields \eqref{LE13}.
	
\end{proof}

\subsubsection{Local existence by Kawashima and Serre}

We have observed in the preceding section that
\begin{equation} \label{LE14}
	\vr(\tau, \cdot) \geq \inf \vr_0 \exp \left( - \int_0^\tau \| \Div \vu \|_{L^\infty(\mathbb{T}^3; R^3)} \dt \right).
\end{equation}	
Similarly, as
\[
c_v \partial_t \vt + c_v \vu \cdot \Grad \vt - \frac{1}{\vr} \kappa \Del \vt \geq \vt \Div\vu,
\]
we get, by the standard comparison theorem,
\begin{equation} \label{LE15}
	\vt(\tau, \cdot) \geq \inf \vt_0 \exp \left( - \frac{1}{c_v} \int_0^\tau \| \Div \vu \|_{L^\infty(\mathbb{T}^3; R^3)} \dt \right).
\end{equation}
Consequently, both the density and the temperature remain bounded below away from zero on any compact subinterval
of $[0, T_{\rm max})$, where $T_{\rm max}$ is the maximal life span for the solution constructed by Cho and Kim.

In view of \eqref{LE14}, \eqref{LE15}, the following local existence result in higher order Sobolev spaces holds, see Serre \cite[Theorem 1.2]{Serr3} and Kawashima and Shizuta \cite{KawShi}.

\begin{Theorem}[{\bf Local existence [Serre]}] \label{LET3}
	Let
	\begin{equation} \label{LE16}
		\vr_0 \in W^{k,2}(\mathbb{T}^3),\ \vr_0 > 0,\
		\vt_0 \in W^{k,2}(\mathbb{T}^3),\ \vt_0 > 0, \
		\vu_0 \in W^{k,2}(\mathbb{T}^3; R^3),
	\end{equation}
	and
	\begin{align}
		\vc{g} &\in  L^2(0, \infty; W^{k-1,2}(\mathbb{T}^3; R^3)),\
		\partial_t \vc{g} \in L^2(0, \infty; W^{k-3}(\mathbb{T}^3; R^3)), \br
		Q &\in L^2(0, \infty; W^{k-1,2} (\mathbb{T}^3)),\
		\partial_t Q \in L^2(0, T; W^{k-3,2}(\mathbb{T}^3)), \ Q \geq 0.
		\label{LE17}
	\end{align}
	for some integer
	\[
	k \geq 3.
	\]
	
	Then there exists $T > 0$ and a strong solution of the Navier--Stokes--Fourier system \eqref{nsf1}--\eqref{nsf7}
	unique in the class
	\begin{align}
		\vr &\in C([0,T]; W^{k,2}(\mathbb{T}^3)),\ \partial_t \vr \in C([0,T]; W^{k-1,2}(\mathbb{T}^3)), \ \vr > 0, \label{LE18} \\
		\vt &\in C([0,T]; W^{k,2}(\mathbb{T}^3)),\
		\vt > 0, \
		\partial_t \vt \in L^2(0,T; W^{k-1,2}(\mathbb{T}^3)), \label{LE19}\\
		\vu &\in C([0,T]; W^{k,2}(\mathbb{T}^3; R^3)) ,\
		\partial_t \vu \in L^2(0,T; W^{k-1,2} (\mathbb{T}^3; R^3)). \label{LE20}	
	\end{align}

\end{Theorem}

\begin{Remark} \label{pop}
	
	As a matter of fact, the results by Serre and Kawashima, Shizuta have been established for the Cauchy problem, meaning the underlying spatial domain is $R^3$. Adapting the proof 
	to the space--periodic setting is straightforward.
	
	\end{Remark}

A short inspection of the proof reveals that the maximal time interval for the local solutions in
Theorem \ref{LET3} is the same for all $k \geq 3$ and therefore coincides with $T_{\rm max}$. In view of Lemma \ref{LEL1}, the maximal existence time of the local solutions in Theorems \ref{LET1}--\ref{LET3} coincide.

Summarizing the above discussion, we get the following version of the local existence for the Navier--Stokes--Fourier system for regular data.

\begin{Theorem}[{\bf Local existence for regular initial data}] \label{LET4}
	
	Let the data belong to the class
	\begin{equation} \label{LE21}
		\vr_0 \in W^{1,q}(\mathbb{T}^3),\ \vr_0 > 0,\
		\vt_0 \in W^{2,2}(\mathbb{T}^3),\ \vt_0 > 0, \
		\vu_0 \in W^{2,2}(\mathbb{T}^3; R^3),
	\end{equation}
	and
	\begin{align}
		\vc{g} &\in BC([0,\infty); L^2(\mathbb{T}^3; R^3))
		\cap L^2(0, \infty; L^q(\mathbb{T}^3; R^3)),\
		\partial_t \vc{g} \in L^2(0, \infty; W^{-1,2}(\mathbb{T}^3; R^3)), \br
		Q &\in BC([0,\infty); L^2(\mathbb{T}^3))
		\cap L^2(0, \infty; L^q(\mathbb{T}^3)),\
		\partial_t Q \in L^2(0, T; W^{-1,2}(\mathbb{T}^3)), \ Q \geq 0,
		\label{LE22}
	\end{align}
	where
	\begin{equation} \label{LE23}
		3 < q \leq 6.
	\end{equation}	
	
	{\bf (i)} Then there exists $0 < T_{\rm max} \leq \infty$ and a strong solution $(\vr, \vt, \vu)$ of the Navier--Stokes--Fourier system unique in the class
	\begin{align}
		\vr &\in C([0,T]; W^{1,q}(\mathbb{T}^3)),\ \partial_t \vr \in C([0,T]; L^q(\mathbb{T}^3)), \ \vr > 0, \label{LE24} \\
		\vt &\in C([0,T]; W^{2,2}(\mathbb{T}^3)) \cap L^2(0,T; W^{2,q}(\mathbb{T}^3)),\ \vt > 0, \br
		\partial_t \vt &\in L^\infty(0,T; L^2(\mathbb{T}^3)) \cap L^2(0,T; W^{1,2}(\mathbb{T}^3)) \label{LE25}\\
		\vu &\in C([0,T]; W^{2,2}(\mathbb{T}^3; R^3)) \cap L^2(0,T; W^{2,q}(\mathbb{T}^3;R^3)), \br
		\partial_t \vu &\in L^\infty(0,T; L^2(\mathbb{T}^3; R^3)) \cap L^2(0,T; W^{1,2}(\mathbb{T}^3; R^3)). \label{LE26}	
	\end{align}
	for any $0 < T < T_{\rm max}$. The maximal time interval $T_{\rm max}$ is bounded below by a positive constant
	depending only on the norm of the data in the spaces \eqref{LE21}--\eqref{LE23}. Moreover,
	\begin{align} \label{LE27}
		T_{\rm max} &< \infty \br &\Rightarrow \br
		\lim_{t \to T_{\rm max} -} &\left( \| \vr(t, \cdot) \|_{W^{1,q}(\mathbb{T}^3)} + \| \vt(t, \cdot) \|_{W^{2,2}(\mathbb{T}^3)} + \| \vu(t, \cdot) \|_{W^{2,2}(\mathbb{T}^3; R^3)}
		+ \|  \vr^{-1}(t, \cdot) \|_{C(\mathbb{T}^3)} \right) = \infty .
	\end{align}
	
	\medskip
	
	{\bf (ii)} If, in addition,
	\begin{equation} \label{LE28}
		\vr_0 \in W^{k,2}(\mathbb{T}^3),\ \vr_0 > 0,\
		\vt_0 \in W^{k,2}(\mathbb{T}^3),\ \vt_0 > 0, \
		\vu_0 \in W^{k,2}(\mathbb{T}^3; R^3),
	\end{equation}
	and
	\begin{align}
		\vc{g} &\in  L^2(0, \infty; W^{k-1,2}(\mathbb{T}^3; R^3)),\
		\partial_t \vc{g} \in L^2(0, \infty; W^{k-3}(\mathbb{T}^3; R^3)), \br
		Q &\in L^2(0, \infty; W^{k-1,2} (\mathbb{T}^3)),\
		\partial_t Q \in L^2(0, T; W^{k-3,2}(\mathbb{T}^3)), \ Q \geq 0.
		\label{LE29}
	\end{align}
	for some integer
	\begin{equation} \label{LE30}
		k \geq 3.
	\end{equation}
	
	Then
	\begin{align}
		\vr &\in C([0,T]; W^{k,2}(\mathbb{T}^3)),\ \partial_t \vr \in C([0,T]; W^{k-1,2}(\mathbb{T}^3)), \ \vr > 0, \label{LE31} \\
		\vt &\in C([0,T]; W^{k,2}(\mathbb{T}^3)),\
		\vt > 0, \
		\partial_t \vt \in L^2(0,T; W^{k-1,2}(\mathbb{T}^3)), \label{LE32}\\
		\vu &\in C([0,T]; W^{k,2}(\mathbb{T}^3; R^3)) ,\
		\partial_t \vu \in L^2(0,T; W^{k-1,2} (\mathbb{T}^3; R^3)). \label{LE33}	
	\end{align}
	for any $0 < T < T_{\rm max}$.
	
\end{Theorem}

\subsubsection{Conditional regularity with regular initial data}

Combining Theorem \ref{BCT1} with Theorem \ref{LET4} and the bounds \eqref{LE14}, \eqref{LE15}, we obtain:

\begin{mdframed}[style=MyFrame]
	
	\begin{Theorem} \label{BCT2} {\bf (Regularity criterion, higher order estimates)}
		
		In addition to the hypotheses of Theorem \ref{BCT1}, suppose that the data belong to the regularity class
		\eqref{LE28}--\eqref{LE30}.
		
		Then there exists a function $\Lambda$, bounded for bounded values of its argument, such that
		\begin{align}
			\sup_{t \in [0,T]} &\| \vr(t, \cdot) \|_{W^{k,2}(\mathbb{T}^3)}	
			\sup_{t \in [0,T]} \| \vt(t, \cdot) \|_{W^{k,2}(\mathbb{T}^3)} +
			\sup_{t \in [0,T]} \| \vu(t, \cdot) \|_{W^{k,2}(\mathbb{T}^3; R^3)} \br
			&+ \sup_{t \in [0,T]} \| \partial_t \vr \|_{W^{k-1,2}(\mathbb{T}^3; R^3)}
			+ \int_0^T \| \partial_t \vt \|_{W^{k-1,2}(\mathbb{T}^3)}^2 \dt + \int_0^T \| \partial_t \vu \|_{W^{k-1,2}(\mathbb{T}^3; R^3)}^2 \dt \br
			&\leq \Lambda \Big( T , \| \vr_0 \|_{W^{k,2}(\mathbb{T}^3)} , \| \vt_0 \|_{W^{k,2}(\mathbb{T}^3)}
			, \| \vu_0 \|_{W^{k,2}(\mathbb{T}^3; R^3)}, 	\br
			&\quad  \int_0^T \| \vc{g} \|^2_{W^{k-1,2}(\mathbb{T}^3; R^3))} \dt,
			\int_0^T \| Q \|^2_{W^{k-1,2}(\mathbb{T}^3))} \dt, \br
			&\quad  \int_0^T \| \partial_t \vc{g} \|^2_{W^{k-3,2}(\mathbb{T}^3; R^3))} \dt,
			\int_0^T \| \partial_t Q \|^2_{W^{k-3,2}(\mathbb{T}^3))} \dt,\br &\quad \left. \sup_{(0,T) \times \mathbb{T}^3} \vr ,
			\sup_{(0,T) \times \mathbb{T}^3} \vt, \left( \inf_{\mathbb{T}^3} \vr_0 \right)^{-1}, \left( \inf_{\mathbb{T}^3} \vt_0 \right)^{-1}, c_v, \mu, \mu^{-1}, \eta, \kappa, \kappa^{-1} \label{BC2a}
			\right) \end{align}
		for any $0 < T < T_{\rm max}$.	
		
	\end{Theorem}

\end{mdframed}

\subsubsection{Stability in higher order topologies}

In view of the refined regularity criterion stated in Theorem \ref{BCT2}, the spaces $X$, $X_T$ as well as the data space $D_{X,F,P}$ can be modified to develop the same theory in stronger topologies. In particular, we get
the stability result stated in Theorem \ref{TST1}. We leave the details to the interested reader.

\section{Statistical solutions}
\label{SE}

We introduce a suitable topology on the phase space $X$ and the corresponding concept of statistical solution
via the  push forward measure argument. To simplify presentation, we denote by
$\vc{U} = (\vr, \vt, \vu)$ the solution of the Navier--Stokes system, $\vc{U}_0 = (\vr_0, \vt_0, \vu_0)$ the initial data, and $(f,p) \in F \times P$ the driving force and the parameter set, respectively.

In particular,
$\vU [\vU_0, f, p] (t, \cdot)$ stands for the value of the strong solution of the NSF system emanating from initial
data $\vU_0$, driven by the external forcing $f$, and with constitutive parameters $p$ evaluated at a time $t$. For the above definition to make sense we may either (i) focus only on $t < T_{\rm max}[\vU_0, f,p]$ or
(ii) define the solution $\vU$ ``after the blow up time'' $T_{\rm max}$. Pursuing the later strategy, we set
\[
\vc{U}(t, \cdot) = \vU_{\infty} \equiv (0,0,0)
\ \mbox{whenever}\ t \geq T_{\rm max}.
\]
At first glance, this may seem a bit surprising as, apparently, the solution is not ``continuous'' for $t \to T_{\rm max}$. However such a choice is convenient as (i) the value $\vU_\infty$ is never reached by any trajectory in finite time (ii) $\vU_0 = \vU_\infty$ is not an admissible initial state so isolated in the family of initial data.

Our goal in the next section is to modify the topology on the phase $X$ in such a way that $\vU \to \vU_\infty$ for
any solution trajectory whenever $t \to T_{\rm max}$, meaning
\[
\left(  \| \vU(t, \cdot) \|_X + \| \vr^{-1}(t, \cdot) \|_{C(\mathbb{T}^3)} \right) \to \infty \ \mbox{as}\ t \to T_{\rm max}-.
\]
To this end, we convert $X$ to a Riemannian manifold with boundary modelled over $X$.

\subsection{Topology on the phase space}

The space $X$ is a reflexive separable Banach space, in particular a Polish metric space. We consider its \emph{open} subset
\[
X^+ = \left\{ (\vr, \vt, \vu) \in X \ \Big|\  \| \vr^{-1} \|_{C(\mathbb{T}^3)} < \infty,\
{\| \vt^{-1} \|_{C(\mathbb{T}^3)} < \infty}  \right\}.
\]
Note carefully that $X^+$ is an \emph{open} subset of the Polish space $X$, whence Polish as well.
Finally, we set
\[
X^+_{\infty} = X^+ \cup \{ \vU_\infty \},\ \vU_\infty = (0,0,0).                                      
\]

Next, we introduce a class of functions
\begin{equation} \label{T1}
	\mathcal{C} = \left\{ \mathcal{F} \in X^+_{\infty} \to R \ \Big|\                      
	\mathcal{F}|_{X^+} \in BC(X^+),\
	\lim_{\| \vU \|_X +  \| \vr^{-1} \|_{C(\mathbb{T}^3)} + {\| \vt^{-1} \|_{C(\mathbb{T}^3)}} \to \infty } \mathcal{F}|_{X^+} (\vU) =
	\mathcal{F}(0) \right\}.
	\end{equation}
Our goal is to define a suitable metrics on the space $X^+_\infty$, for which the class $\mathcal{C}$ will coincide with $BC(X_{\infty}^+).$
To this end, consider the functions -- the Fourier coefficients -- defined as
\[
F_{\vc{k}}(\vc{U}) = \int_{\mathbb{T}^3} \vc{U} \cdot e_{\vc{k}} \ \dx,\ e_\vc{k} \ \mbox{are normed trigonometric polynomials}, \ \vc{k} \in Z^5,\ \vU \in X,
\]                                                                                                                                  
together with $G : X^+_\infty \to R$, 	
\[
G (0) \equiv G(\vU_\infty) = 0 , \ G (\vU) = \left( 1 +  \| \vU \|_X + \| \vr^{-1} \|_{C(\mathbb{T}^3)}                      
+ {\| \vt^{-1} \|_{C(\mathbb{T}^3)} < \infty}  \right)^{-2} \ \mbox{if}\
\vU = (\vr, \vt, \vu) \in X^+ .
\]
Finally, set
\[
Q_{i, \vc{k}}(\vU)  = \left\{                                                                                                                                             
\begin{array}{l} G ( \vU ) \ \mbox{if} \ i = 1, \\ \\                             
	G ( \vU  ) F_{\vc{k}} (\vU) \ \mbox{if}\  i = 2. \end{array} \right.	
\]	
Observe that $Q_{i, \vc{k}} \in \mathcal{C}$, $i=1,2$, $\vc{k} \in Z^5$.

 \begin{Proposition}[{\bf Topology on the phase space $X^+_\infty$}]
 Let
 \begin{equation} \label{T2}
 	d( \vc{U}; \vc{V} ) = \sum_{ i =1,2 , \vc{k} \in Z^5 } \exp (- |\vc{k}|) \frac{ \left| Q_{i,\vc{k} } (\vU) -
 		Q_{i,\vc{k} } (\vc{V}) \right| }{ 1 + \left| Q_{i,\vc{k} } (\vU) -
 		Q_{i,\vc{k} } (\vc{V}) \right| }.
 	\end{equation}

 Then $d$ defines a metric on $X^+_\infty$ and $(X^+_\infty; d)$ is a Polish space. In addition,
 the class $\mathcal{C}$ coincides with $BC(X^+_\infty)$ -- the space of bounded continuous functions on $X^+_\infty$.
 
 \begin{Remark} 
 	
 	Note carefully that the topology on the space $X^+_\infty$ \emph{is not} a one point compactification of the space $X^+$ in the sense of Alexandroff.                                                                                                                            
 	
 	\end{Remark}

\end{Proposition}	 	

\begin{proof}
	
{\bf Step 1:}

To see $d$ is a metric, it is enough to observe that
\[
d( \vc{U}; \vc{V} ) = 0 \ \Leftrightarrow \ \vU = \vc{V}.
\]
First observe
\[
G(\vU) = 0 \ \Leftrightarrow \ \vU = 0.
\]
Thus if $\vU \ne 0$ and $d( \vc{U}; \vc{V} ) = 0$, then necessarily $\vc{V} \ne 0$, and $F_\vc{k}(\vU) = F_{\vc{k}} (\vc{V})$ for all $\vc{k}$, which yields the desired conclusion.

{\bf Step 2:}

Next, we show the equivalence
\begin{equation} \label{T3}
d(\vU_n; \vU) \to 0 \ \mbox{and}\ \vc{U} \in X^+ \ \Leftrightarrow \ \vU_n \in X^+ \ \mbox{for all}\ n \ \mbox{large enough, and}\
\vU_n \to \vU \ \mbox{in}\ X.	
	\end{equation}
Indeed
\[
G(\vU_n) \to G(\vU) =  \left( 1 +  \| \vU \|_X + \| \vr^{-1} \|_{C(\mathbb{T}^3)} +
{\| \vt^{-1} \|_{C(\mathbb{T}^3)}}  \right)^{-2} > 0
\]	
yields $\vU_n \in X^+$ for all $n$ large enough. In particular,
\[
\left( 1 +  \| \vU_n \|_X + \| \vr^{-1}_n \|_{C(\mathbb{T}^3)} + {\| \vt_n^{-1} \|_{C(\mathbb{T}^3)}}  \right)^{-2} \to \left( 1 +  \| \vU \|_X + \| \vr^{-1} \|_{C(\mathbb{T}^3)} + {\| \vt^{-1} \|_{C(\mathbb{T}^3)}}  \right)^{-2},
\]
in other words
\begin{equation} \label{T4}
\| \vU_n \|_X + \| \vr^{-1}_n \|_{C(\mathbb{T}^3)} +
{\| \vt_n^{-1} \|_{C(\mathbb{T}^3)}} \to  \| \vU \|_X + \| \vr^{-1} \|_{C(\mathbb{T}^3)}
+ {\| \vt^{-1} \|_{C(\mathbb{T}^3)}}.
\end{equation}
Thus it follows
\begin{equation} \label{T8}
	F_{\vc{k}} (\vU_n) \to F_{\vc{k}}(\vU) \ \mbox{for any}\ \vc{k} \in Z^5,
\end{equation}
which, together with \eqref{T4}, yields
\begin{equation} \label{T9}
	\vU_n \to \vU \ \mbox{weakly in}\ X.
\end{equation}
Since the phase space $X$ is compactly embedded in $C(\mathbb{T}^3; R^5)$ we deduce, using \eqref{T4},
\[
 \| \vr^{-1}_n \|_{C(\mathbb{T}^3)} \to \| \vr^{-1} \|_{C(\mathbb{T}^3)},\
 \| \vt^{-1}_n \|_{C(\mathbb{T}^3)} \to \| \vt^{-1} \|_{C(\mathbb{T}^3)};
\]
whence
\begin{equation} \label{T10}
\| \vU_n \|_X \to \| \vU \|_X.
\end{equation}
As $X$ is a reflexive Banach space, relations \eqref{T9}, \eqref{T10} yield the desired conclusion
\[
\vU_n \to \vU \ \mbox{in}\ X.
\]
The reverse implication in \eqref{T3} is trivial.

{\bf Step 3:}

We show the equivalence
\begin{align} \label{T11}
	d(\vU_n; \vU) &\to 0,\ \vU = 0 \br &\Leftrightarrow \br \mbox{for any}\ M &> 0 \ \mbox{there exists}\ n(M) \ \mbox{such that for all}\ n \geq n(M) \br \mbox{either}\
\vU_n &= 0 \ \mbox{or}\ 	\| \vU_n \|_X + \| \vr^{-1}_n \|_{C(\mathbb{T}^3)} +
{\| \vt^{-1}_n \|_{C(\mathbb{T}^3)}}
> M	
\end{align}

Indeed $d(\vU_n, 0 ) \to 0$ implies $G(\vU_n) \to 0$ yielding the left-right implication. The reverse implication is similar. 	

{\bf Step 4:}

On the one hand, given \eqref{T3}, \eqref{T11}, it is easy to show that $\mathcal{C} \subset BC(X^+_\infty)$. On the other hand, if $\mathcal{F} \in BC(X^+_\infty)$
and $\vU$, the equivalence \eqref{T3} shows that $\mathcal{F}$ is continuous at $\vU$ in the $d-$metric. Moreover, we have
\[
\mathcal{F} (\vU_n) \to \mathcal{F} (0) \ \mbox{whenever}\ d(\vU_n, 0 ) \to 0.
\]
In view \eqref{T3}, this necessarily means
\[
\lim_{\| \vU \|_X +  \| \vr^{-1} \|_{C(\mathbb{T}^3)} \to \infty } \mathcal{F}|_{X^+} (\vU) =
\mathcal{F}(0);
\]
whence $\mathcal{F} \in \mathcal{C}$.

{\bf Step 5:}

Next we show that the metric space $(X^+_\infty;d)$ is complete. To this end consider a Cauchy sequence $(\vU_n)_{n = 1}^\infty$. First suppose there is a subsequence such that
\begin{equation} \label{T12}
\vU_{n(k)} \in X^+,\
\| \vU_{n(k)} \|_X + \| \vr^{-1}_{n(k)} \|_{C(\mathbb{T}^3)}
+ {\| \vt^{-1}_{n(k)} \|_{C(\mathbb{T}^3)} } \to Y \in (0, \infty).
\end{equation}
As the space $X$ with the original topology is complete, we conclude that
\[
\vU_{n(k)} \to \vU \ \mbox{in}\ X,\ \vU \in X^+.
\]
However, as the sequence is Cauchy, we infer $\vU_n \to \vU$ in $X$.

If \eqref{T12} does not hold, we have the alternative \eqref{T11} yielding
$d(\vU_n; 0) \to 0$.

{\bf Step 6:}

As the space $X$ is separable, the corresponding dense countable set augmented by $\vU_\infty = (0,0,0)$ is dense countable in $(X^+_\infty; d)$.

	\end{proof}

\subsection{Continuity of the solution operator in time}

At this stage, we may extend the solution $\vU$ beyond the existence time $T_{\rm max}$ by setting
\begin{equation} \label{T12s}
\vU [ \vU_0; f; p] (t, \cdot) = \left\{ \begin{array}{l} \mbox{the strong solution of the NSF system it}\
t < T_{\rm max}[\vU_0; f; p] , \\  \vU_\infty = (0,0,0) \ \mbox{if}\ t \geq T_{\rm max},\ \vU_0 \in X^+, \\
\vU_\infty = (0,0,0) \ \mbox{if}\ \vU_0 = \vU_\infty.
 \end{array} \right. 	
	\end{equation}
It is easy to check that the mapping
\begin{equation} \label{T13}
	t \in [0, \infty) \mapsto \vU [ \vU_0; f; p] (t, \cdot) \in X^+_\infty
\end{equation} 	
is continuous with values in $X_\infty^+$ \emph{endowed with the metric $d$} introduced in \eqref{T2} for any admissible data.

\subsection{Continuity of the solution operator with respect to the data}

We examine the continuity of the extended operator with respect to the data $D_n$. Consider
\[
D_n \to D \ \mbox{in}\ D_{X,F,P}
\]
with the corresponding solution
\[
\vU [D_n] (t, \cdot) \ \mbox{at a specific time}\ t > 0.
\]
Suppose first that $t < T_{\rm max}[D].$ Then, in view of Theorem \ref{TST1},
\begin{equation} \label{T14}
	\vU[D_n] (t, \cdot) \to \vU [D] (t, \cdot) \ \mbox{in}\ C(\mathbb{T}^3; R^5).
	\end{equation}

\subsection{Autonomous problem and the semigroup property}
\label{APP}

Suppose that the functions $\vc{g}$ and $Q$ are independent of time. Accordingly, the data space can be written
as
\begin{align}
	D_{X,F,P} = X \times F \times P,\ X &= W^{1,q}(\mathbb{T}^3) \times  W^{2,2}(\mathbb{T}^3) \times W^{2,2}(\mathbb{T}^3; R^3), \br
	F &=  L^q(\mathbb{T}^3; R^3) \times L^q(\mathbb{T}^3),\ P = R^4
	\nonumber
\end{align}

Now, it is easy to check that the operator $\vU$ defined through \eqref{T12s} enjoys the semigroup property:
\begin{equation} \label{sep}
\vU [ \vU_0; f; p] (t + s, \cdot) = \vU \Big[ \vU[ \vU_0 ; f;p] (s, \cdot) ;f;p \Big] (t, \cdot),\ s, t \geq 0.	
	\end{equation}

\subsection{Statistical solutions}

Suppose we are given a complete Borel probability measure $\mathcal{V}$ on the data space $D_{X,F,P} = X^+ \times F \times P$.
We define the push forward measure
\begin{equation} \label{PF}
\int_{X^+_\infty} \mathcal{F}(\vU_0) \D \mathcal{V}_t (\vU_0) =
\int_{D_{X,F,P}} \mathcal{F}(\vU [D] (t, \cdot)) \ \D \mathcal{V} (D) \ \mbox{for any}\ \mathcal{F} \in \mathcal{C},  	
	\end{equation}
where $\vU$ is the extended solution operator defined by \eqref{T12s}. Note that
\begin{align} \label{PF1}
	\int_{X^+_\infty} \mathcal{F}(\vU_0) \D \mathcal{V}_t (\vU_0) &=
	\int_{D_{X,F,P}} \mathcal{F}(\vU [D] (t, \cdot)) \ \D \mathcal{V} (D) \br
	& \int_{D_{X,F,P}} \mathcal{F}(\vU [D] (t, \cdot)) \mathds{1}_{t < T_{\rm max}[D] } \D \mathcal{V} (D) +
		\mathcal{F}(0) \mathcal{V} \{ t \geq T_{\rm max} \}	.
\end{align}
Thus $(\mathcal{V}_t )_{t \geq 0}$ is a family of Borel probability measures on $X^+_\infty$ with $\mathcal{V}_0 = \Pi_{\mathcal{X}^+_\infty} \mathcal{V}$.

\begin{mdframed}[style=MyFrame]
	
	\begin{Definition}[{\bf Statistical solution}] \label{SSD1}
		
		The family of Markov operators
		\[
		\mathcal{M}_t : \mathfrak{P}[D_{X,F,P}] \to \mathfrak{P}[X^+_\infty],
		\]
		defined as
		\[
		\mathcal{M}_t (\mathcal{V} ) = \int_{X^+ \times F \times P}
		\delta_{\vU[D](t, \cdot) } \D \mathcal{V}(D),\ t \geq 0
		\]
		for any $\mathcal{V}$ supported by admissible data is called
		\emph{statistical solution}
		of the NSF system.
		
		In addition, the Markov operators $\mathcal{M}_t$  are regular with the \emph{dual operators}
		\[
		\mathcal{M}^*_t :  BC(X^+_\infty) \to \mathcal{B}(D_{X,F,P}) \ \mbox{(Borel functions on the data space)}
		\]
		\[
		\mathcal{M}^*_t : \mathcal{F} \mapsto \Big[ D \mapsto \mathcal{F}( \vU [D] (t, \cdot)) \Big].
		\]
 There holds
	\[
	\int_{X^+_\infty} \mathcal{F} \ \D \mathcal{M}_t(\mathcal{V})  = \int_{D_{X,F,T}} \mathcal{M}^*_t (\mathcal{F}) \  \D \mathcal{V},\ t \geq 0.
	\]	
	\end{Definition}
	
\end{mdframed}

\begin{Remark}
Note that  $\mathcal{M}_t(\mathcal{V})\{\mathbf{U}_\infty\} = \mathcal{V}\{t > T_{\rm max}\},$ meaning the measure of the data set 
for which the solution blows up before reaching $t$..
\end{Remark}

\begin{Remark}

	Conformally with the existing theory of \emph{stationary statistical solutions} (see e.g.
	\cite{FoRoTe1}), one can define the statistical solutions as measures on the \emph{trajectory space}
	\[
	C([0, \infty); X^+_\infty),
	\]
	\[
	\mathcal{M}_t (\mathcal{V}) = \int_{X^+ \times F \times P} \delta_{\vc{U}[D]} \ \D \mathcal{V}(D)
	\in \mathfrak{P}[C([0, \infty); X^+_\infty)].
	\]

	\end{Remark}

\subsubsection{Disintegration, autonomous problem, semigroup property}

We focus on the situation described in Section \ref{APP}, where the forcing terms $\vc{g}$ and $Q$ are independent of $t$.
First we use the Disintegration Theorem to write
\[
\int_{D_{X,F,P}} \mathcal{F}(D) \ \D \mathcal{V}(D)
= \int_{F \times P} \left( \int_X \mathcal{F} (\vU_0; f;p) \D [ \mathcal{V}|(f,p) ] (\vU_0)   \right) \D \Pi_{F \times P} \mathcal{V}(f,p)
\]
where $\mathcal{V}^{f,p} = \mathcal{V} | (f,p)$ is a family of Borel probability measures on $X$.
Given $(f,p) \in F \times P$, the push forward measure  $\mathcal{V}^{f,p}_t$ is defined on data phase space $X^+_{\infty}$ as
\begin{align} \label{SE1A}
	\int_{X^+_{\infty}} &\mathcal{F} (\vU_0) \ \D
	\mathcal{V}^{f,p}_t (\vU_0)  \br &=
	\int_X \mathcal{F} \Big(\vU [\vU_0; f; p] (t, \cdot) \Big)\mathds{1}_{ \{ t < T_{\rm max}[\vU_0; f; p] \} }
	\ \D [ \mathcal{V} | (f,p) ] (\vU_0) +
	\mathcal{F}(0) \mathcal{V} \{ t \geq T_{\rm max} \}.	
\end{align}
for any $\mathcal{F} \in \mathcal{C}$.

It follows from the semigroup property of the solution operator established in Section \ref{APP} that the exists a semigroup of Markov operators
\[
( \mathcal{M}^{f,p}_t )_{t \geq 0}: \mathfrak{P}[X^+_\infty] \to \mathfrak{P}[X^+_\infty],
\
\mathcal{V}^{f,p}_t = \mathcal{M}^{f,p}_t ( [\mathcal{V}|(f,p)] ) .
\]
Finally, we assume that the initial data and the forcing/parameter data are independent, meaning
\[
[\mathcal{V}|(f,p)] = \Pi_X \mathcal{V},\ \mathcal{V} = \Pi_X \mathcal{V} \otimes \Pi_{F \times P} \mathcal{V}.
\]
This leads to the following result:

\begin{mdframed}[style=MyFrame]

\begin{Theorem} [{\bf Autonomous problem, semigroup property}] \label{TTT1}
	
	Suppose that $\vc{g}$, $Q$ are independent of $t$ and that the measure $\mathcal{V}$ defined on the data space
	$D_{X,F,P} = X \times F \times P$  and supported by  admissible data decomposes as
\[
\mathcal{V} = \Pi_X \mathcal{V} \otimes \Pi_{F \times P} \mathcal{V}.
\]

Then the statistical solution introduced in Definition \ref{SSD1} can be written in the form
\begin{equation} \label{markov}
\mathcal{M}_t ( \mathcal{V} ) = \int_{F \times P} \mathcal{M}^{f,p}_t ( \Pi_X \mathcal{V} ) \ \D \Pi_{F \times P} \mathcal{V} (f,p) ,\ t \geq 0,
	\end{equation}
where $(\mathcal{M}^{f,p}_t)_{t \geq 0}$ is a semigroup of Markov operators:

\begin{itemize}
	\item
	\[
	 \mathcal{M}^{f,p}_t : \mathfrak{P}[X^+_\infty] \to \mathfrak{P}[X^+_\infty],
	 \]
\item
	\[
	\mathcal{M}^{f,p}_t (\delta_{\vU_0} ) = \delta_{\vU[\vU_0, f,p ](t, \cdot)}, \ t \geq 0,
	\]	
 \item
	 \[
	 \mathcal{M}^{f,p}_0 (\nu) = \nu, \ \mathcal{M}^{f,p}_{t + s}(\nu) = \mathcal{M}^{f,p}_{t} ( \mathcal{M}^{f,p}_{s}(\nu)),\ s,t \geq 0,
	 \]
	 \item
	 \[
	 \mathcal{M}^{f,p}_t \left( \sum_{{i=1}}^n \lambda_i \nu_i \right) = \sum_{{i=1}}^n \lambda_i \mathcal{M}^{f,p}_t \left( \nu_i \right),\ \lambda_i \geq 0,\ \sum_{{i=1}}^n \lambda_i = 1, \ t \geq 0,
	 \]
	\item
	\[
	t \mapsto \mathcal{M}^{f,p}_t (\nu) \ \mbox{is continuous in the narrow topology of}\  \mathfrak{P}[X^+_\infty].
	\]

	\end{itemize}

\end{Theorem}

\end{mdframed}

\section{Applications, Monte Carlo approximation of statistical solution}
\label{MC}

We consider the Monte Carlo type approximation of statistical solutions with random initial data leaving the forcing and parameters $(f,p)$ deterministic. We suppose that
$(\vr_0, \vt_0, \vu_0)$ are admissible random initial data - random variables on a probability space
$(\Omega, \mathcal{B}, \Bbb{P})$.

Consider $N$ independent identically distributed (i.i.d.) copies of the data
$(\vr_0, \vt_0, \vu_0)$ denoted
$(\vr_0^n, \vt^n_0, \vu^n_0)_{n = 1}^N$. To each data we associate the extended solution introduced in \eqref{T12s}, $(\vr^n, \vt^n, \vu^n)[\vr_0^n, \vt^n_0, \vu^n_0; f; p](t, \cdot)$ -- i.i.d. random variables
for any $t \geq 0$.

Suppose that
\[
\expe{ \int_{\mathbb{T}^3} \vr_0 \dx } < \infty,\
\expe{ \int_{\mathbb{T}^3} \left( \frac{1}{2} \vr_0 |\vu_0|^2 + c_v \vr_0 \vt_0 \right)  \ \dx}
< \infty,\  - \expe{ \int_{\mathbb{T}^3} \vr_0 \log \left( \frac{\vt^{c_v}_0}{\vr_0} \right) \dx }
< \infty,
\]
where $\mathbb{E}$ stands for expected value. In view of the {\it a priori} bounds \eqref{conM} -- \eqref{entrP},
we get
\begin{align} \label{bounds}
\expe{ \| \vr(t, \cdot) \|_{L^1(\mathbb{T}^3)} \mathds{1}_{t < T_{\rm max}[ \vr_0, \vt_0, \vu_0;f;p ] } } &< \infty,\br
\expe{ \| \vr \vu (t, \cdot) \|_{L^1(\mathbb{T}^3; R^3)} \mathds{1}_{t < T_{\rm max}[ \vr_0, \vt_0, \vu_0;f;p ] } }
&< \infty,\br  \expe{ \int_{\mathbb{T}^3} \left\| \vr \log \left( \frac{\vt^{c_v}}{\vr} \right) (t, \cdot) \right\|_{L^1(\mathbb{T}^3)} \mathds{1}_{t < T_{\rm max}[ \vr_0, \vt_0, \vu_0;f;p ] } }
&< \infty.
\end{align}

Applying the Banach space variant of Strong Law of Large Numbers, see Ledoux, Talagrand \cite[Chapter 7, Corollary 7.10]{LedTal}, we conclude
\begin{align} \label{SLLN}
\frac{1}{N}	\sum_{n=1}^N \vr^n (t, \cdot) \mathds{1}_{t < T_{\rm max}[ \vr^n_0, \vt^n_0, \vu^n_0;f;p ] }
	&\to \expe{ \vr (t, \cdot) \mathds{1}_{t < T_{\rm max}[ \vr_0, \vt_0, \vu_0;f;p ]} }
	\ \mbox{in}\ L^1(\mathbb{T}^3) \ \Bbb{P} \mbox{ a.s.}, \br
\frac{1}{N}	\sum_{n=1}^N \vr^n \vu^n (t, \cdot) \mathds{1}_{t < T_{\rm max}[ \vr^n_0, \vt^n_0, \vu^n_0;f;p ] }
&\to \expe{ \vr \vu (t, \cdot) \mathds{1}_{t < T_{\rm max}[ \vr_0, \vt_0, \vu_0;f;p ]} }
\ \mbox{in}\ L^1(\mathbb{T}^3; R^3) \ \Bbb{P} \mbox{ a.s.}, \br
\frac{1}{N}	\sum_{n=1}^N \vr^n \log \left( \frac{(\vt^n)^{c_v}}{\vr^n} \right) (t, \cdot) \mathds{1}_{t < T_{\rm max}[ \vr^n_0, \vt^n_0, \vu^n_0;f;p ] }
&\to \expe{ \vr \log \left( \frac{\vt^{c_v}}{\vr} \right) (t, \cdot)  \mathds{1}_{t < T_{\rm max}[ \vr_0, \vt_0, \vu_0;f;p ]} }
\br &\mbox{in}\ L^1(\mathbb{T}^3) \ \Bbb{P} \mbox{ a.s.}	
	\end{align}

To justify the last step in the context of the present theory, we have to observe
that the mapping
\[
\mathcal{G}: (\vr, \vt, \vu) \mapsto \left\{ \begin{array}{l} \left( \vr, \vr \vu, \vr \log \left( \frac{\vt^{c_v}}{\vr} \right)
	\right) \ \mbox{if}\ (\vr, \vt, \vu) \in X^+ ,\ \vt > 0 , \\ \\
	0 \ \mbox{if}\ (\vr, \vt, \vu) = (0,0,0)
\end{array} \right.
\]
is Borel measurable from $X^+_\infty$ to $L^1(\mathbb{T}^3) \times L^1(\mathbb{T}^3) \times L^1(\mathbb{T}^3, R^3)$.
Indeed if
\[
G_n \in BC[0, \infty), \ 0 \leq G_n \leq 1, \ G_n(Y) = Y \ \mbox{for}\ Y \leq n,
\]
then
\[
G_n \left( \| \vU \|_X + \| \vr^{-1} \|_{C(\mathbb{T}^3)} + \| \vt^{-1} \|_{C(\mathbb{T}^3)} \right) \mathcal{G} (\vU) \in BC (X^+_\infty; (L^1)^3) \to
\mathcal{G}(\vU) \ \mbox{for any}\ \vU \in X^+ .
\]

\section*{Conflict of interest}
\noindent The authors declare that they have no conflict of interest.

\section*{Data availability}
\noindent Data sharing not applicable to this article as no data sets were generated or analysed during the current study.




\def\cprime{$'$} \def\ocirc#1{\ifmmode\setbox0=\hbox{$#1$}\dimen0=\ht0
	\advance\dimen0 by1pt\rlap{\hbox to\wd0{\hss\raise\dimen0
			\hbox{\hskip.2em$\scriptscriptstyle\circ$}\hss}}#1\else {\accent"17 #1}\fi}

\end{document}